\documentclass[10pt]{amsart}

\usepackage{amsmath,amsfonts,amssymb,amsthm,mathtools}

\usepackage{verbatim}
\usepackage{amsmath,amsfonts}
\usepackage[mathscr]{euscript} 
\usepackage{amsthm}
\usepackage{url}

\usepackage{graphicx} 

\usepackage{enumitem} 

\usepackage{hyperref} 
\hypersetup{colorlinks} 

\usepackage[colorinlistoftodos]{todonotes}

\RequirePackage{cleveref}
\usepackage{hypcap}
\hypersetup{colorlinks=true, citecolor=darkblue, linkcolor=darkblue}
\definecolor{darkblue}{rgb}{0.0,0,0.7}
\newcommand{\darkblue}{\color{darkblue}}
\newcommand{\defn}[1]{\emph{\darkblue #1}}

\setlist[enumerate]{
	label=\textnormal{({\roman*})},
	ref={\roman*}}

\makeatletter
\def\th@plain{%
	\thm@notefont{}
	\itshape 
}
\def\th@definition{%
	\thm@notefont{}
	\normalfont 
}
\makeatother

\makeatletter
\def\fdsy@scale{1}
\newcommand\fdsy@mweight@normal{Book}
\newcommand\fdsy@mweight@small{Book}
\newcommand\fdsy@bweight@normal{Medium}
\newcommand\fdsy@bweight@small{Medium}

\DeclareFontFamily{U}{FdSymbolB}{}
\DeclareFontShape{U}{FdSymbolB}{m}{n}{
	<-7.1> s * [\fdsy@scale] FdSymbolB-\fdsy@mweight@small
	<7.1-> s * [\fdsy@scale] FdSymbolB-\fdsy@mweight@normal
}{}
\DeclareFontShape{U}{FdSymbolB}{b}{n}{
	<-7.1> s * [\fdsy@scale] FdSymbolB-\fdsy@bweight@small
	<7.1-> s * [\fdsy@scale] FdSymbolB-\fdsy@bweight@normal
}{}

\makeatother

\DeclareSymbolFont{fdrelations}{U}{FdSymbolB}{m}{n}
\SetSymbolFont{fdrelations}{bold}{U}{FdSymbolB}{b}{n}

\DeclareMathSymbol{\lescc}{\mathrel}{fdrelations}{66}

\newtheorem{thm}{Theorem}
\newtheorem{lemma}[thm]{Lemma}

\newtheorem{cor}[thm]{Corollary}

\newtheorem{conv}[thm]{Convention}

\theoremstyle{definition}
\newtheorem{ex}[thm]{Example}
\newtheorem{rem}[thm]{Remark}

\def\wh{\widehat}

\def\sq{\square}
\def\zz{\mathbb Z}

\def\Ga{\Gamma}
\def\Om{\Omega}

\def\ga{\gamma}

\def\al{\alpha}
\def\be{\beta}
\def\om{\omega}

\def\ssu{\subset}

\def\<{\langle}
\def\>{\rangle}

\def\rB{ {\text {\rm B} } }
\def\rD{ {\text {\rm D} } }

\def\LE{ {\text {\rm LE} } }

\def\0{{\mathbf 0}}

\def\.{\hskip.06cm}
\def\ts{\hskip.03cm}

\def\lra{\leftrightarrow}

\def\La{\Lambda}

\def\ze{{\zeta}}

\def\nin{\noindent}

\def\Pb{{\text{\bf P}}}

\def\aF{\textrm{F}}
\def\aFr{\textrm{\em F}}

\def\aN{\textrm{N}}
\def\aNr{\textrm{\em N}}

\title{Log-concavity in planar random walks}
\date{\today}

\author{Swee Hong Chan}
\address[Swee Hong Chan]{Department of Mathematics, UCLA,  Los Angeles, CA 90095.}
\email{\texttt{sweehong@math.ucla.edu}}

\author[\ts Igor Pak]{Igor Pak}
\address[Igor Pak]{Department of Mathematics, UCLA,  Los Angeles, CA 90095.}
\email{\texttt{pak@math.ucla.edu}}

\author[\ts Greta Panova]{Greta Panova}
\address[Greta Panova]{Department of Mathematics, USC,  Los Angeles, CA 90089.}
\email{\texttt{gpanova@usc.edu}}

\begin{document}

\begin{abstract}
We prove log-concavity of exit probabilities of lattice random walks
in certain planar regions.
\end{abstract}

\maketitle

\section{Introduction}\label{sec:intro}
In the study of \emph{random walks}, there is a fundamental idea
based on its \emph{Markovian property}:  when some ``life event''
happens to the walk, the future trajectory of the walk can be changed,
and this transformation can be exploited to obtain both quantitative
and qualitative conclusions on the random walk distributions.  These
``life events'' can be rather mundane, for example the first time
when the walk returns to the starting point, crosses with some other random walk,
enters an obstacle, etc.  On the other hand, the conclusions can be
quite remarkable and include the classical \emph{reflection principle},
and the \emph{Karlin--McGregor formula} also known as the
\emph{Lindstr\"om--Gessel--Viennot lemma} in the discrete setting,
see~$\S$\ref{ss:finrem-hist}.

In this paper we use a variation on this approach for random walks
in simply connected planar regions.  The conclusion is qualitative
in some sense that at the end we prove \emph{log-concavity} of a
certain natural  exit probability distribution.  We chose to present
a discrete version of the result rather than the (somewhat cleaner)
continuous version as the former is more powerful and amenable to
generalizations (see Section~\ref{sec:gen}), while the latter follows
easily by taking limits.

\medskip

\begin{thm}[{\rm Special case of Theorem~\ref{t:trans}}]\label{t:main}
Let \ts $\Ga\ssu \zz^2$ \ts be a simply connected region in the plane
whose boundary \. $\partial \Ga = \al\cup \eta_+\cup \eta_- \cup \be$ \.
is comprised of two vertical intervals \ts $\al, \ts \be$,
and two \ts $x$-monotone lattice paths \ts $\eta_+, \ts \eta_-$\ts,
see Figure~\ref{f:region}.
We assume that \ts $\al \ssu \{x=0\}$ \ts and \ts $\be \ssu \{x=m\}$ \ts
for some $m>0$.

Let \ts $\{X_t\}$ \ts be the nearest neighbor lattice random walk  which
starts at the origin \ts $X_0=O \in \al$, and is absorbed
whenever \ts $X_t$ \ts tries to exit the region~$\ts\Ga$.  Denote by \ts $T$ \ts the
first time~$t$ such that \ts $X_t \in \be$, and let \ts $p(k)$ \ts be
the probability that $X_T=(m,k)$.  Then \ts $\{p(k)\}$ \ts is log-concave:
$$p(k)^2 \, \ge \, p(k+1) \. p(k-1) \qquad \text{for all \ $k\in \zz$, \ such that \ \ $(m,k\pm 1)\ts \in \be$}.
$$
\end{thm}

\begin{figure}[hbt]
	\includegraphics[width=6.cm]{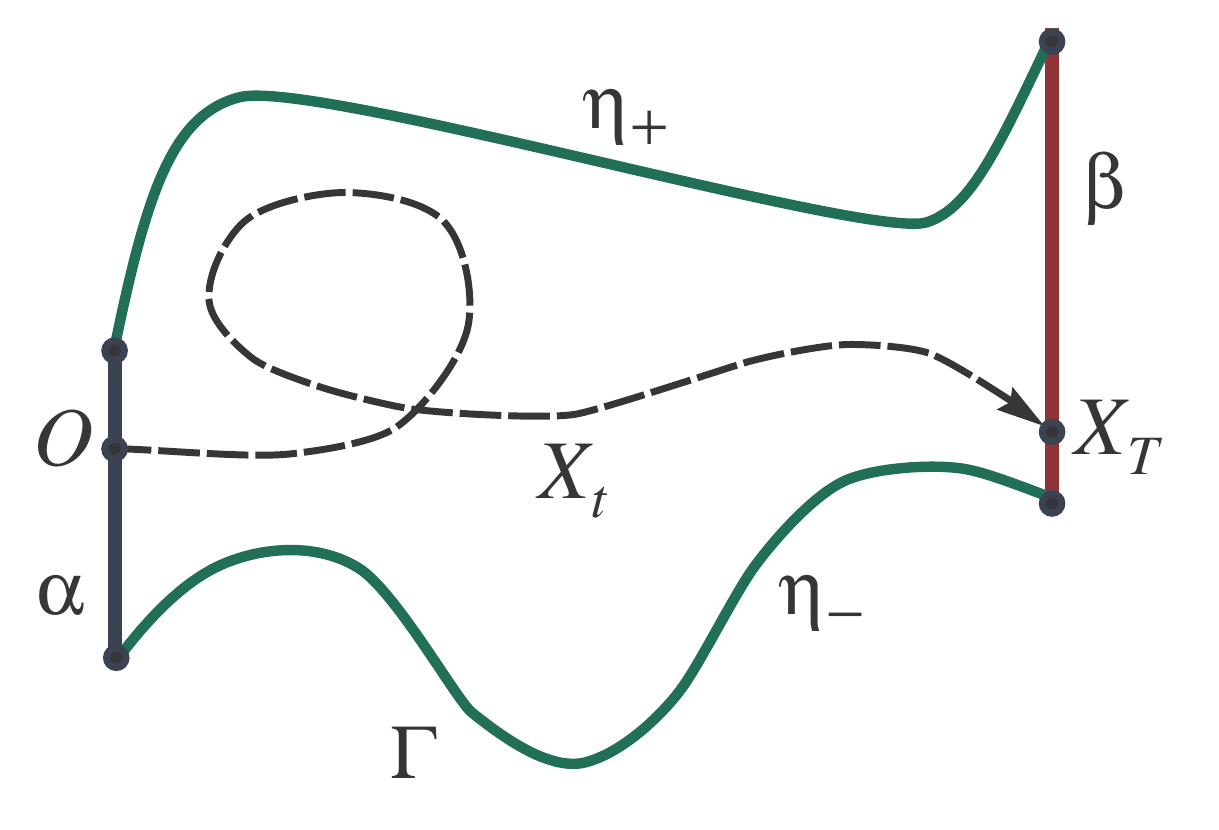}
\vskip-0.3cm
	\caption{Random walk $X_t$ in a region $\Ga$ as in the theorem.}
	\label{f:region}
\end{figure}


In particular, the theorem implies that the sequence \ts $\{p(k)\}$ \ts is
\emph{unimodal} (see e.g.~\cite{Bra}).  This also points to the difficulty
of proving the result by a direct calculation, as in general there is no
natural point at which the probability maximizes.  We refer to \ts $p(k)$
\ts as \defn{exit probabilities}, since one can think of them as probabilities
the walks exits the region through different points on the interval~$\beta$.

Perhaps surprisingly, Theorem~\ref{t:main} is a byproduct of our recent
work~\cite{CPP2} on the \emph{Stanley inequality} for the case of
posets of width two;  in fact, we obtain the inequality as a corollary
of the theorem (see~$\S$\ref{ss:finrem-posets}).

\bigskip

\section{Proof of Theorem~\ref{t:main}}

We start with the following combinatorial result.  Throughout this section,
a \emph{lattice path} is a path in~$\zz^2$, possibly self-intersecting
along vertices and edges.

\medskip

\begin{lemma}\label{l:main}
In the conditions of Theorem~\ref{t:main}, let \ts $A,B\in \al$ \ts and \ts $C,C',D,D'\in \be$ \ts be
six points on the boundary of the region~$\Ga$, such that
$$
A=(0,a), \ B=(0,b), \ C= (m,c), \ C'=(m,c-r), \ D'=(m,d+r), \ D=(m,d),
$$
and suppose \ts $(a-b) \leq  (c-d)-r$, where \ts $r>0$.  Then there is an injection \.
$$\Phi: \,\bigl\{(\xi_{AC},\xi_{BD})\bigr\} \, \longrightarrow \, \bigl\{(\xi_{AC'},\xi_{BD'})\bigr\},$$
where

\smallskip

$\circ$ \ $\xi_{AC}: A\to C$, \. $\xi_{BD}: B\to D$, \. $\xi_{AC'}: A\to C'$, and \. $\xi_{BD'}: B\to D'$ \. are lattice paths which lie inside~$\Ga$,

\smallskip

$\circ$ \ the sum of numbers of horizontal edges in \ts $\xi_{AC}$ \ts and \ts $\xi_{BD}$ \ts which project onto \ts $[i,i+1]$, \ts is equal to

\hskip.45cm
the sum of numbers of horizontal edges in \ts $\xi_{AC'}$ \ts and \ts $\xi_{BD'}$ \ts which project onto \ts $[i,i+1]$, for all \. $0\le i \le m-1$,

\smallskip

$\circ$ \ the sum of numbers of vertical edges in \. $\xi_{AC}$ \. and \. $\xi_{BD}$ \. which project onto \ts $j$, \ts is equal to

\hskip.45cm
the sum of numbers of vertical edges in \. $\xi_{AC'}$ \. and \. $\xi_{BD'}$ \. which project onto \ts $j$, for all \. $0\le j \le m$.
\end{lemma}


\smallskip

Here by \emph{project} we mean the vertical projection onto the $x$-axis.
By adding the edge and vertex count equalities above over all~$i$ and~$j$,
we conclude that injection $\ts \Phi$ \ts
preserves the total length of these paths:
$$\bigl|\xi_{AC}\bigr| \. + \. \bigl|\xi_{BD}\bigr| \, = \,
\bigl|\xi_{AC'}\bigr| \. + \. \bigl|\xi_{BD'}\bigr|\ts.
$$


\begin{proof}[{Proof of Lemma~\ref{l:main}}]
The proof is an explicit construction of the injection~$\Phi$,
illustrated in Figure~\ref{f:simpllistic}.  See also Figure~\ref{f:steps}
in the next section for a detailed construction of each step.

\smallskip

\nin
$(0)$  \. We start with the lattice paths \ts $\xi_{AC}$ \ts and
\ts $\xi_{BD}$ \ts drawn inside the region~$\Ga$.  Note that by the definition
of the points in the lemma, we have \ts $|CC'| = |D'D|=r$.

\smallskip

\nin
$(1)$  \. Let \ts $\ell:=b+c-d-r$ \ts and let \ts $B':=(0,\ell)$. By the assumption in the lemma,
$B'$ lies above $A$ on the line spanned by~$\al$. Denote by \ts $\wh\eta_-$ \ts the lattice
path \ts $B\to D$ \ts formed by following interval~$\alpha$ down from~$B$,
then path \ts $\eta_-$ \ts and then interval~$\be$ up to~$D$.
Let \ts $\chi$ \ts be the lattice path
obtained by shifting  \ts $\eta_-$ \ts up at distance~$(\ell-b)$.
Similarly, let \ts
$\wh \chi: B'\to C'$ \ts be the lattice path
obtained by shifting  \ts $\wh\eta_-$ \ts up at distance~$(\ell-b)$.

Note that the path \ts $\xi_{AC}$ \ts starts below \ts $\wh\chi$ \ts and ends
above \ts $\wh\chi$.  Thus \ts $\xi_{AC}$ \ts intersects \ts $\wh\chi$ \ts at least once, where
the intersection points could be multiple and include~$A$.
Order these points of intersection according to the order in which they appear
on~$\ts\xi_{AC}$, and denote by $E$ the last such point of intersection.  Finally,
denote by \ts $\xi_{EC}$ \ts the last part of the path \ts $\xi_{AC}$ \ts
between~$E$ and~$C$, and note that \ts $\xi_{EC}$ \ts lies \emph{above} the
$x$-monotone lattice path~$\wh\chi$.

\smallskip

\nin
$(2)$  \. Denote by \ts $\wh \eta_+$ \ts the lattice path \ts $A\to C$, starting at $A$, following $\al$ up, then \ts$\eta_+$ and ending by following $\be$ down to $C$. 
Let \ts $\zeta_{B'C'}: B'\to C'$ \ts  be the lattice path
obtained by shifting  \ts $\xi_{BD}$ \ts up at distance~$\ell-b$.  By the same argument
as above, path \ts $\zeta_{B'C'}$ \ts start above and ends below \ts $\wh \eta_+$.
Denote by $F$ the last point of intersection of  \ts $\wh \eta_+$ \ts and
\ts $\zeta_{B'C'}$ \ts according to the order on which they appear on~$\zeta_{B'C'}$\ts.
Finally, denote by \ts $\ze_{FC'}$ \ts the last part of the path \ts $\zeta_{B'C'}$ \ts
between~$F$ and~$C$, and note that \ts $\zeta_{FC'}$ \ts lies \emph{below} the
$x$-monotone lattice path~$\eta_+$\ts.

\smallskip

\nin
(3) \. Observe that \ts $\zeta_{B'C'}$ \ts lies above \ts $\wh\chi$ \ts since shifting
both paths down gives \ts $\xi_{BD}$ \ts lying above \ts $\eta_-$, respectively.
Also, the path \ts $\xi_{EC}$ \ts lies below \ts $\eta_+$ \ts and above \ts $\wh\chi$ \ts by definition.
Since $C'$ is below $C$, we have that $E$ and $C$ lie on different sides of $\zeta_{FC'}$.
Thus lattice paths \ts $\xi_{EC}$ \ts and \ts $\zeta_{FC'}$ \ts must intersect
in the connected component $\La$ of the region between \ts $\eta_+$ \ts and~$\wh\chi$
that contains interval $[CC']\ssu \be$.  There could be many such intersections,
of course, including multiple intersections when the paths form loops.

\smallskip

\begin{lemma}[{\rm Fomin~\cite[Thm~6.1]{Fom}}]\label{l:Fomin}
Let \ts $\ga_1: E\to C$ \ts and \ts $\ga_2: F\to C'$ \ts be two intersecting paths between
boundary points of the simply connected region \ts $\La \subseteq \Ga$.
Let \ts $G \in \ga_1\cap \ga_2$ \ts be an intersection point, and suppose
$$
\ga_1 \. := \. E \. \to_{\ga_1'} \. G \. \to_{\ga_1''} \. C \qquad \text{and} \qquad
\ga_2 \. := \. F \. \to_{\ga_2'} \. G \. \to_{\ga_2''} \. C\ts,
$$
by which we mean that \ts $G$ \ts separates \ts $\ga_i$ \ts into two paths: \ts $\ga_i'$
\ts and \ts $\ga_i''$, where \ts $i\in \{1,2\}$.
Then there is a well defined \defn{key intersection} point \ts $G$ \ts as above,
such that the map \ts $\bigl\{(\ga_1,\ga_2)\bigr\}\to \bigl\{(\ga_1^\ast,\ga_2^\ast)\bigr\}$ \ts
is an injection, where \. the pair of paths \ts
$(\ga_1^\ast,\ga_2^\ast)$ \ts is obtained from \ts $(\ga_1,\ga_2)$ \ts by
a swap at~$G$~:
$$\ga_1^\ast \. := \. E \. \to_{\ga_1'} \. G \. \to_{\ga_2''} \. C' \qquad \text{and} \qquad
\ga_2^\ast \. := \. F \. \to_{\ga_2'} \. G \. \to_{\ga_1''} \. C\ts.
$$
\end{lemma}

\smallskip

The lemma is a special case of the (much more general) result by Fomin;
below we include  a proof sketch for completeness.
Denote by $G$ the \ts {key intersection} \ts of paths
\ts $\xi_{EC}$ \ts and \ts $\zeta_{FC'}$ \ts defined by the lemma.
Finally,
denote by \ts $\xi_{GC}$ \ts the last part of the path \ts $\xi_{EC}$ \ts
between~$G$ and~$C$, and
 by \ts $\ze_{GC'}$ \ts the last part of the path \ts $\zeta_{FC'}$ \ts
between~$G$ and~$C$.

\begin{figure}[hbt]
\includegraphics[width=12.3cm]{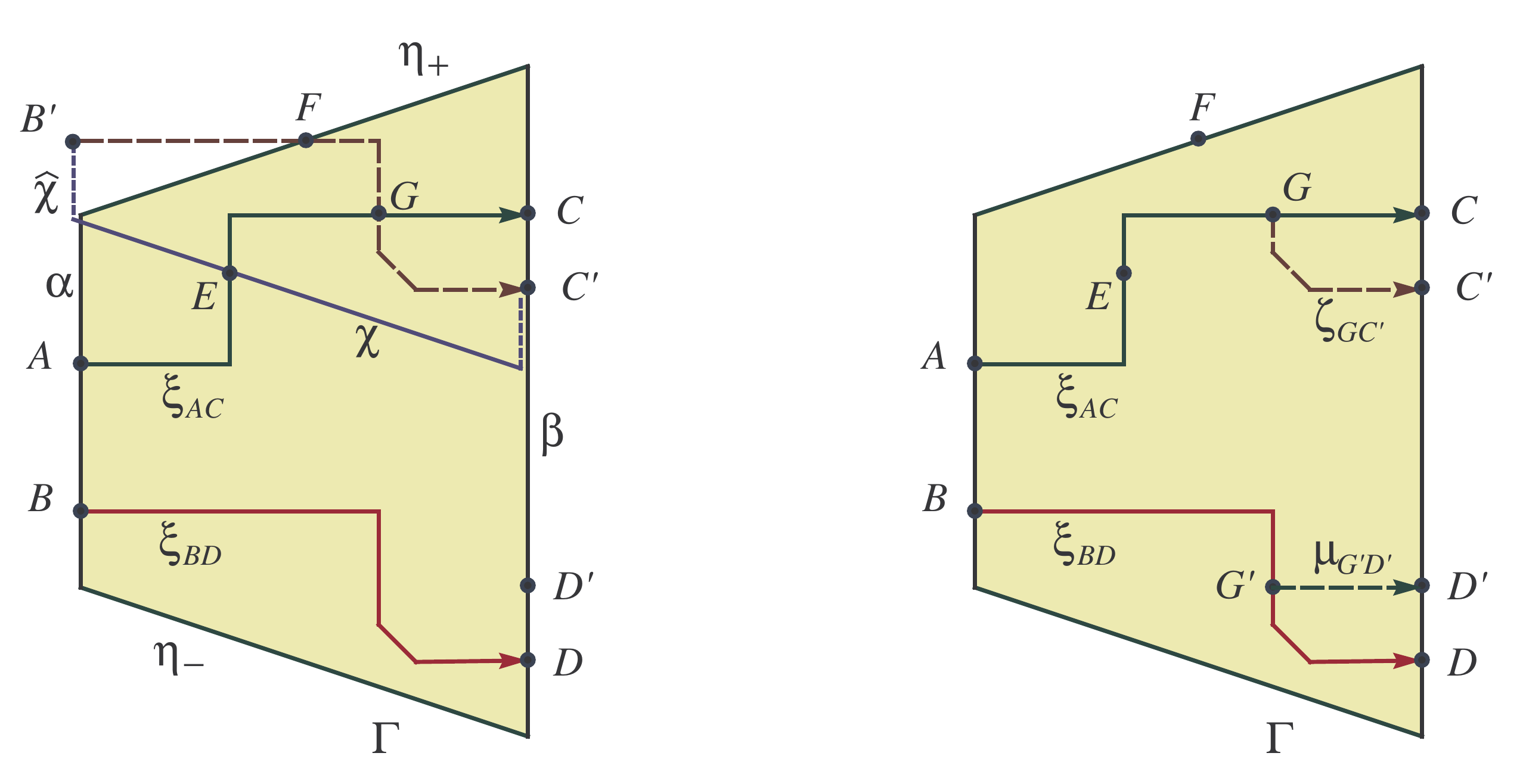}
	\caption{Construction of paths \ts $\zeta_{GC'}$ \ts
and \ts $\mu_{G'D'}$ \ts in the proof of the lemma.}
	\label{f:simpllistic}
\end{figure}


\smallskip

\nin
(4) \. Denote by \ts $\xi_{AG}: A \to G$ \ts the lattice path \ts $\xi_{AC}$ \ts  without the
last part \ts $\xi_{GC}$.  In the notation of the lemma, define \ts $\xi_{AC'}$ \ts to be
the path \ts $\xi_{AG}$ \ts followed by \ts $\ze_{GC'}$.

\smallskip

\nin
(5) \. Let $G'$ be the point on~$\xi_{BD}$ \ts obtained by shifting $G$ down at distance~$\ell$, and denote
by \ts $\xi_{G'D}$ \ts the last part of the path \ts $\xi_{BD}$ \ts
between~$G$ and~$D$.  Similarly, let \ts $\mu_{G'D'}$ \ts be the path obtained shifting
down at distance~$\ell$ the path \ts $\ze_{GC'}\ts$.
Denote by \ts $\xi_{BG'}: B \to G'$ \ts the lattice path \ts $\xi_{BD}$ \ts  without the
last part \ts $\xi_{G'D}$.
In the notation of the lemma, define \ts $\xi_{BD'}$ \ts to be
the path \ts $\xi_{BG'}$ \ts followed by \ts $\mu_{G'D'}$.
%
%

\medskip

\nin
{\bf Claim:} \ts {\emph The map \. $\Phi: \ts \big(\xi_{AC},\xi_{BD}\big) \ts\to \ts \big(\xi_{AC'},\xi_{BD'}\big)$ \.
constructed above is an injection.}

\medskip

To prove the claim, we consider the inverse of $\Phi^{-1}$. Start with a pair of lattice paths \ts $(\xi_{AC'},\xi_{BD'})$ \ts
and follow the steps as above after relabeling \ts $C\lra C'$, $D\lra D'$.
This construction will not work at all cases as the shifted paths are no longer
guaranteed to intersect because of topological considerations.  However,
when \. $(\xi_{AC'},\xi_{BD'}) = \Phi(\xi_{AC},\xi_{BD})$, the construction
will work for the same reason and produce the pair of lattice paths \ts
$(\xi_{AC},\xi_{BD})$ \ts as in the steps~$(0)$--$(5)$.  Here we are using the
key intersection point in step~$(3)$ to ensure the construction
is well defined and can be inverted at this step.
 We see that $\Phi$ is an involution between pairs of lattice paths $\bigl\{(\xi_{AC},\xi_{BD})\bigr\}$ and the pairs of paths $\bigl\{(\xi_{AC'},\xi_{BD'})\bigr\}$, which intersect as in~$(3)$ after the translations in~$(1)$ and~$(2)$. The details are straightforward.

\smallskip

Finally, the projection conditions on~$\Phi$ as in the lemma are straightforward.
Indeed, we effectively swap parts of lattice paths:  \. $\xi_{GC}$ \ts with
\ts $\ze_{GC'}$, and \. $\xi_{G'D}$ \ts with
\ts $\mu_{G'D'}$.  Since path \ts $\xi_{GC}$ \ts is shifted path \ts $\mu_{G'D'}$,
and path \ts $\xi_{G'D}$ \ts is shifted path \ts $\ze_{GC'}$, this implies both conditions.
\end{proof}

\medskip

\begin{proof}[Proof of Theorem~\ref{t:main}]
In the notation of Lemma~\ref{l:main}, set \ts $a=b=0$, so both points
\ts $A=B$ \ts lie at the origin.  Further, set \ts $r=1$, $c=k+1$ and $d=k-1$,
so \ts $C=(m,k+1)$, \ts $C'=D'=(m,k)$ \ts and \ts $D=(m,k-1)$. In this case,
the injection \ts $\Phi$ \ts shows that the number of pairs of lattice
paths \ts $O\to (m,k+1)$ \ts and \ts  $O\to (m,k-1)$, is less of equal than
the number of lattice paths \ts $O\to (m,k)$, squared.

We are not done, however, as our paths overcount the paths
implied by the probabilities in the theorem, since we consider only \emph{the first time}
by the lattice random walk~$X_t$ is at~$\be$.  Recall that
$\Phi$ preserves the number of horizontal edges which project onto point~$m$
and onto~$[m-1,m]$.  The assumption in the theorem that $T$ is
the first time the walk is at~$\be$ can be translated as having exactly one
of point $(m,\ast)$ and exactly one edge \ts $(m-1,\ast)\to (m,\ast)$ \ts corresponding
to the last step of the walk~$X_t$.  Therefore, this property is also preserved
under~$\Phi$.  This completes the proof.   \end{proof}

\smallskip

\begin{rem}{\rm The level of generality in Lemma~\ref{l:main} may seem like an overkill
as we only use a special case in the proof of the theorem.  In reality, explaining
the special case needed for Theorem~\ref{t:main} is no easier than the general case
in the lemma.  In fact, setting \ts $A=B$ \ts only makes it more
difficult to keep track of the paths.  Furthermore, other properties in the
lemma are used heavily in the next section.  }\end{rem}

\medskip

\begin{proof}[Sketch of proof of Lemma~\ref{l:Fomin}]
Let \ts $\Om\ssu \zz^2$ \ts be a simply connected region and let \ts
$\ga: P\to Q$ be a lattice walk in~$\Om$, where \ts $P,Q\in \Om$.
Define a \defn{loop-erased walk}
\ts $\LE(\ga)$ \ts by removing cycles as they appear in~$\ga$.  Formally,
take the first self-intersection point~$X$, where \. $\ga: P\to_{\ga'} X \to_{\ga''} X \to_{\ga'''}  Q$,
and remove part~$\ga''$.  Repeat this until the resulting path \ts $\LE(\ga)$ \ts has no
self-intersections.

Suppose \ts $P_1,Q_1,Q_2,P_2\in \partial \Om$ \ts are points on the boundary
(oriented clockwise) and in this order.  For a pair of paths \ts $(\ga_1,\ga_2)$,
$\ga_1: P_1\to Q_2$, $\ga_2: P_2\to Q_1$, note that \ts $\ga_1$ \ts and
\ts $\ga_2$ \ts intersect by planarity, and so do \ts $\LE(\ga_1)$ \ts and \ts $\ga_2$.
Let \ts $X$ \ts be the intersection point of \ts $\ga_2$ \ts and \ts $\LE(\ga_1)$ \ts
which lies closest to~$P_1$ along the path \ts $\LE(\ga_1)$.
Note that point~$X$ can appear multiple times on \ts $\ga_1\supseteq \LE(\ga_1)$.

Consider the point \ts $X\in \ga_1$ \ts such
that the edge \ts $Y\to X$ \ts in~$\ga_1$ is not deleted in~$\LE(\ga_1)$.
Similarly, choose the first~$X$ on~$\ga_2$.  This defines
the \defn{key intersection} of paths \ts $\ga_1$ \ts and \ts $\ga_2$.
As in the lemma, swap the future of these paths to obtain paths \ts
$(\ga_1^\ast,\ga_2^\ast)$.
To see that the map \. $(\ga_1,\ga_2) \ts \to \ts (\ga_1^\ast,\ga_2^\ast)$ \.
in an injection, note that it is invertible  on all pairs \ts $(\ga_1^\ast,\ga_2^\ast)$
such that \ts $\LE(\ga_1^\ast)$ \ts intersects~$\ga_2^\ast$.  We omit the details.
\end{proof}

\smallskip

\begin{rem}{\rm There is a much larger probabilistic context in which
the \emph{loop-erased random walk} plays a prominent role, see
e.g.~\cite[$\S$11]{LL} and~$\S$\ref{ss:finrem-hist}.
}\end{rem}

\bigskip

\section{Large example and subtleties in the proof}
The construction in the proof above may seem excessively complicated at first,
given that the map \ts $\Phi$ \ts is easy to define in the example in
Figure~\ref{f:simpllistic}.  Indeed, in that case one can simply shift \ts
$\xi_{BD}$ \ts up \ts to define path \ts $\zeta_{B'D'}$, find the last
(only in this case) intersection point~$G$ with \ts $\xi_{AC}$, and swap
the futures of these paths as we do in~$(4)$ and~$(5)$.  Voila!

\smallskip

Unfortunately, this simplistic approach does not work for multiple
reasons.  Let's count them here:

\medskip

\nin
{\small (\textit{i})} \ts Path \ts $\zeta_{GC'}$ \ts does not have to be
inside~$\Ga$.  This is why we defined point~$F$ in~$(2)$.

\smallskip

\nin
{\small (\textit{ii})} \ts Path \ts $\zeta_{G'D'}$ \ts does not have to be
inside~$\Ga$.  This is why we defined path~$\chi$ and point~$E$ in~$(1)$.

\smallskip

\nin
{\small (\textit{iii})} \ts Paths \ts $\chi$  \ts and \ts $\xi_{AC}$ \ts
do not have to intersect at all.  This is why we defined path
\ts \ts $\wh \chi$ \ts in~$(1)$.

\smallskip

\nin
{\small (\textit{iv})} \ts Paths \ts $\zeta_{B'C'}$ \ts and \ts $\xi_{AC}$ \ts
do have to intersect for geometric reason, since \ts $(a-b) \leq  (c-d)-r$.
On the other hand, paths \ts $\xi_{EC}$ \ts and \ts $\zeta_{FC'}$ \ts intersect
for topological reasons.  This is why in~$(3)$ we consider a connected
component of~$\La$ between paths \ts $\eta_+$ \ts and \ts $\wh\chi$.  Note that
the latter can in fact intersect, possibly multiple times.

\smallskip

\nin
{\small (\textit{v})} \ts Paths \ts $\zeta_{B'C'}$ and \ts $\xi_{AC}$ \ts
 can have
multiple loops intersecting each other in multiple ways.  There is no
natural way to define the ``last intersection'' which would be easy to
reverse.  This is why in~$(3)$  we invoked Lemma~\ref{l:Fomin},
whose proof requires loop-erased walk and symmetry breaking.\footnote{In the first
draft of the paper we were not aware of the issue \ts {\small (\textit{v})}, only to discover
it when lecturing on the result.}

\medskip


To help the reader understand the issues \ts {\small (\textit{i})},  \ts {\small (\textit{ii})}
and \ts {\small (\textit{iv})}, consider a large example in Figure~\ref{f:steps}.
Here paths \ts $\xi_{AC}$ \ts and \ts $\chi$ \ts intersect multiple times in step~$(1)$
defining~$E$.  Then, in step~$(2)$, path \ts $\zeta_{B'C'}$ \ts intersects the boundary
\ts $\eta_+$ \ts multiple times.  Note that~$F$ is defined as the last intersection
along path \ts $\zeta_{B'C'}$, not along \ts $\eta_+$\ts.  The same property applies
to~$E$, even if in the example it is the last intersection on both paths.

What exactly goes wrong in \ts {\small (\textit{v})} \ts in the definition of~$G$
is rather subtle and we leave this as an exercise to the reader.
Note that there is no issue similar to \ts {\small (\textit{v})} \ts in
the definition of points~$E$ and~$F$, since the boundaries \ts $\eta_\pm$ \ts
are $x$-monotone.

\medskip

\begin{rem}
Note that we explicitly use $x$-monotonicity of the boundary paths \ts $\eta_\pm$\ts.
In~$\S$\ref{ss:gen-boundary} below, we address what happens when the boundary is
not $x$-monotone.
\end{rem}

\begin{figure}[hbt]
\includegraphics[width=7.2cm]{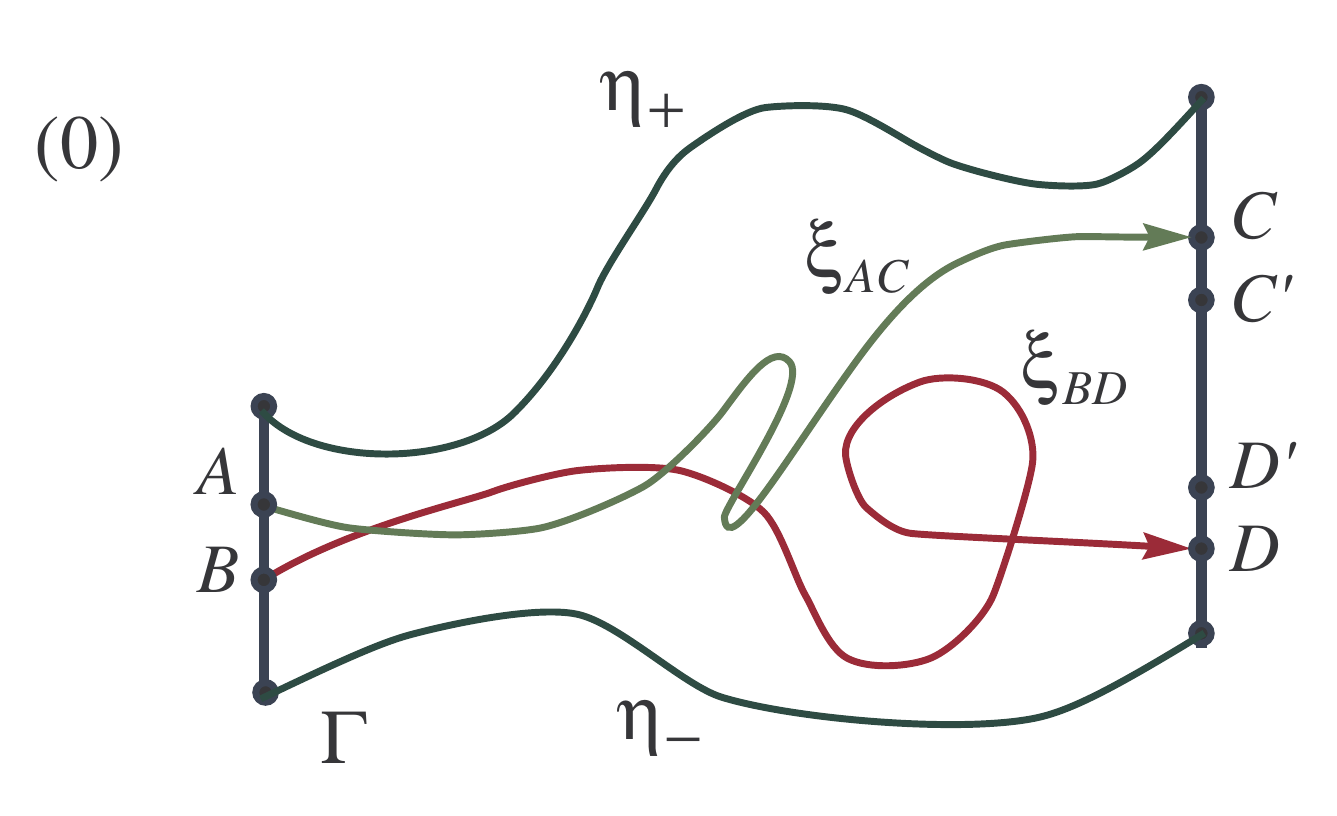} \qquad \includegraphics[width=7.2cm]{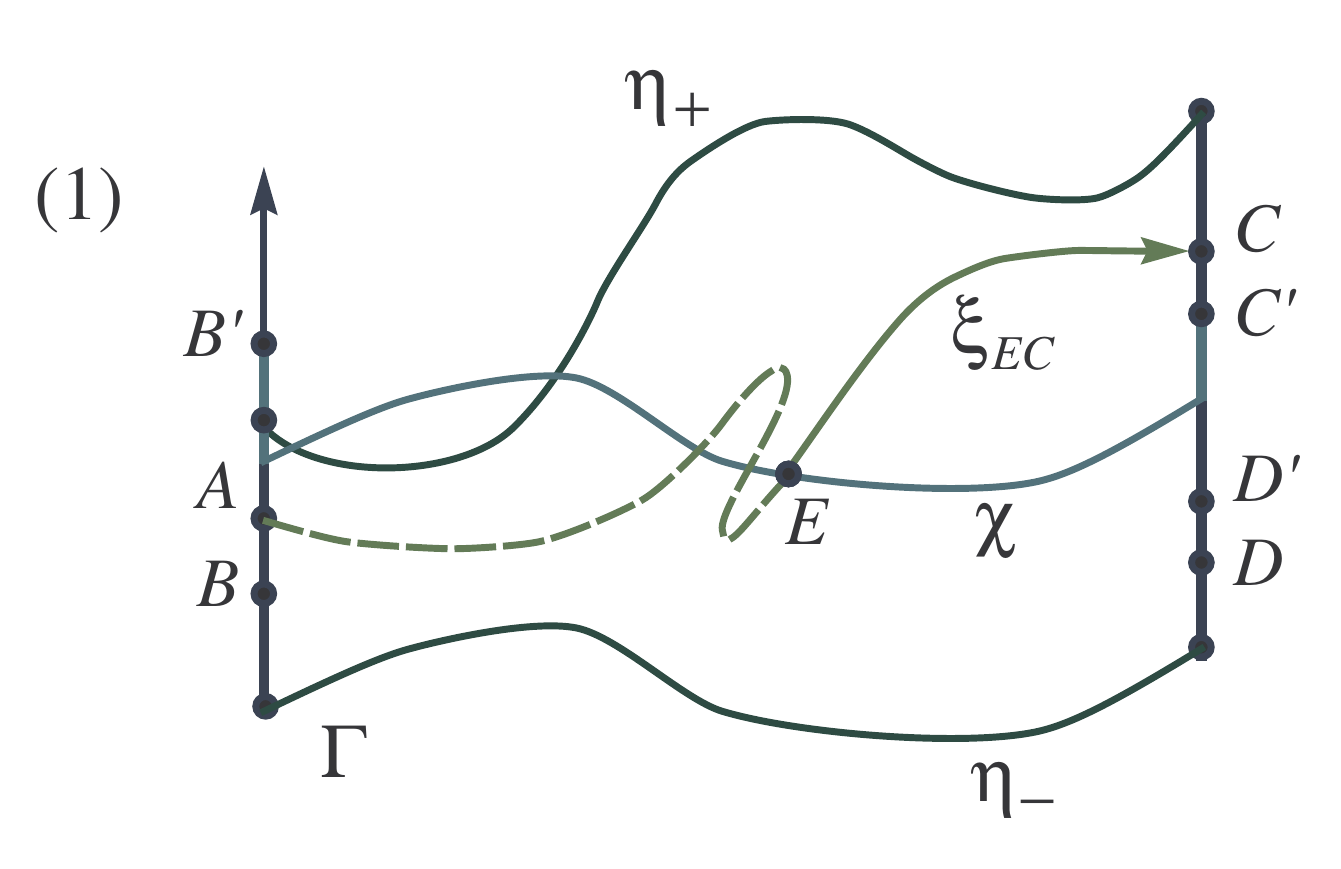} \\
\includegraphics[width=7.2cm]{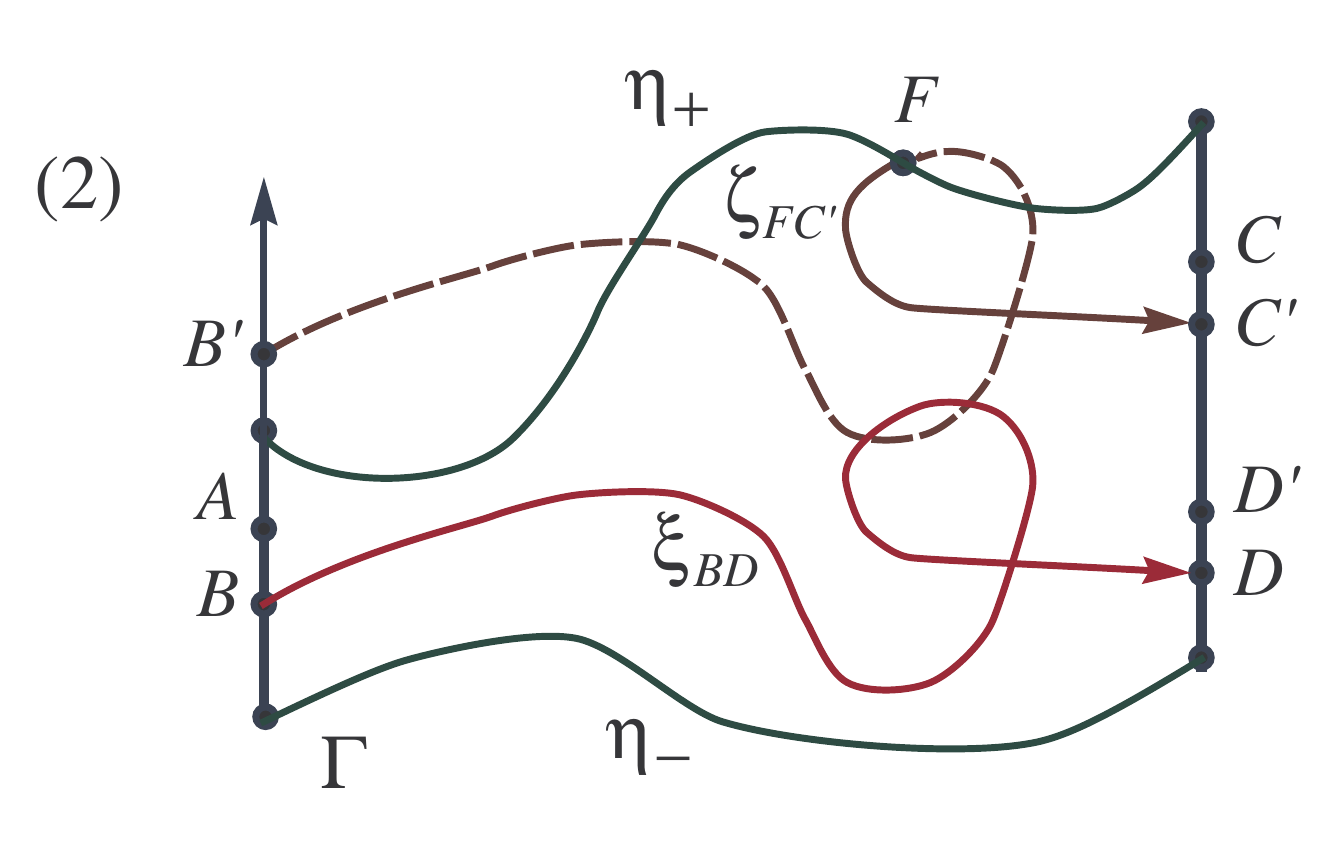}\qquad \includegraphics[width=7.2cm]{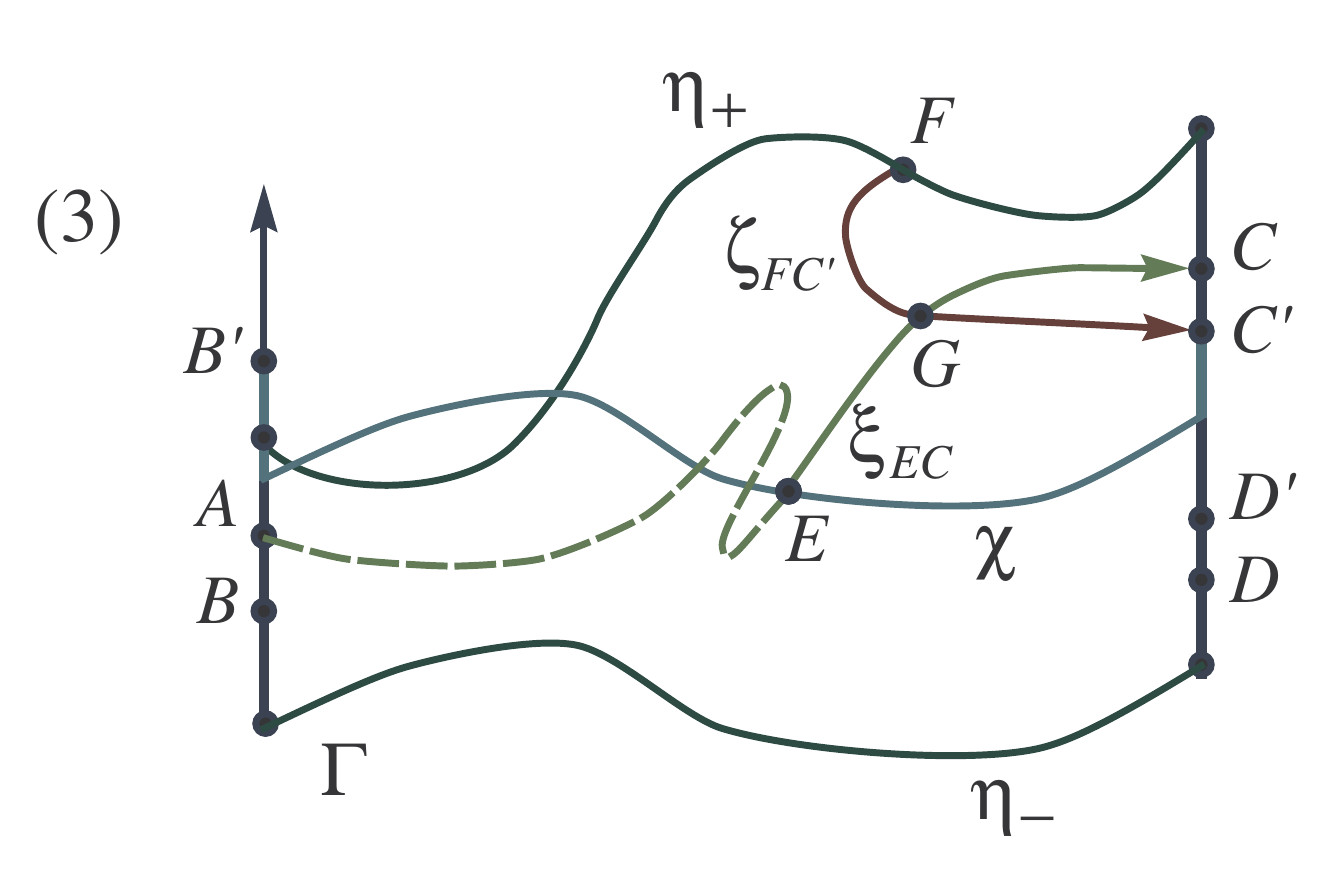} \\
\includegraphics[width=7.2cm]{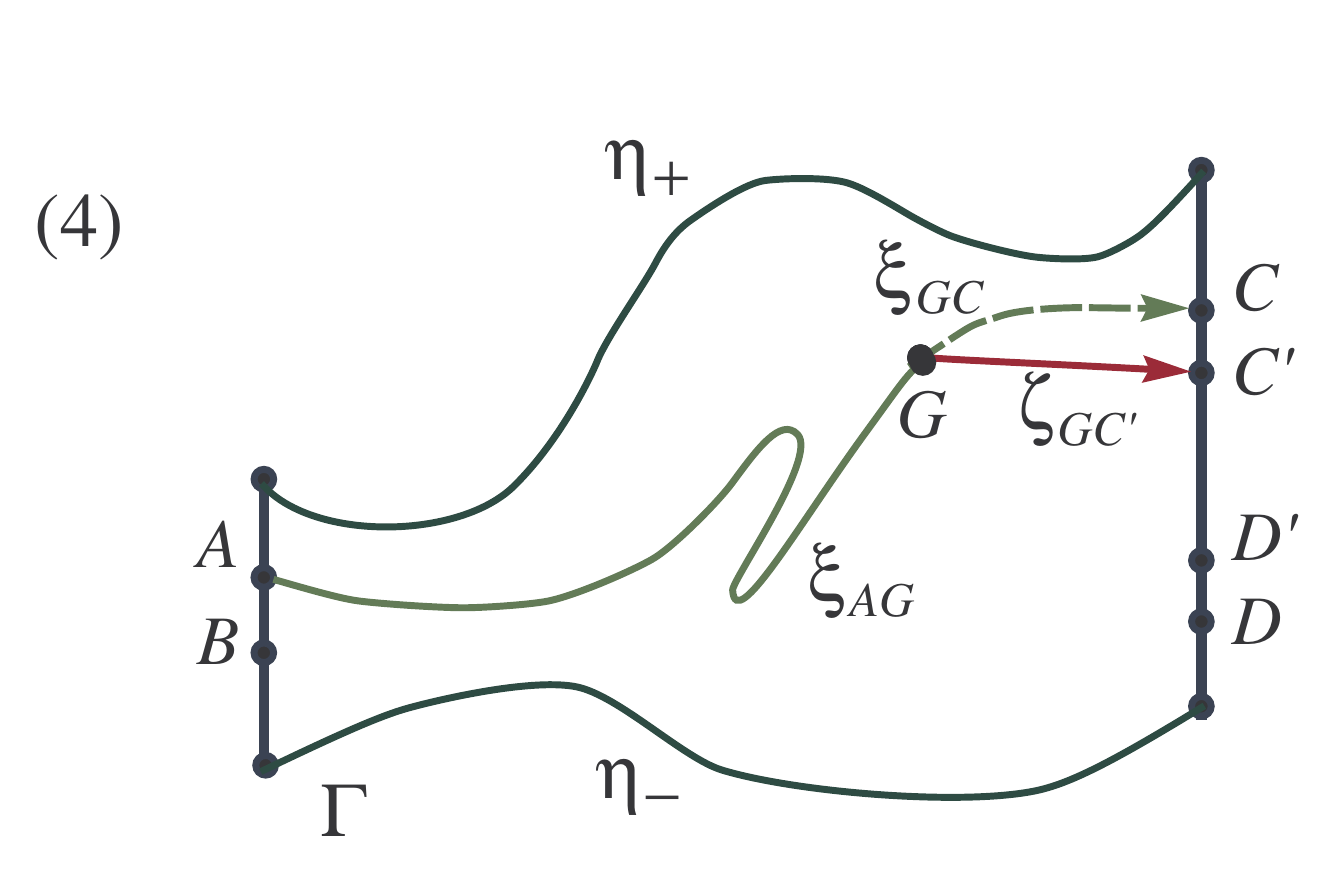}\qquad \includegraphics[width=7.2cm]{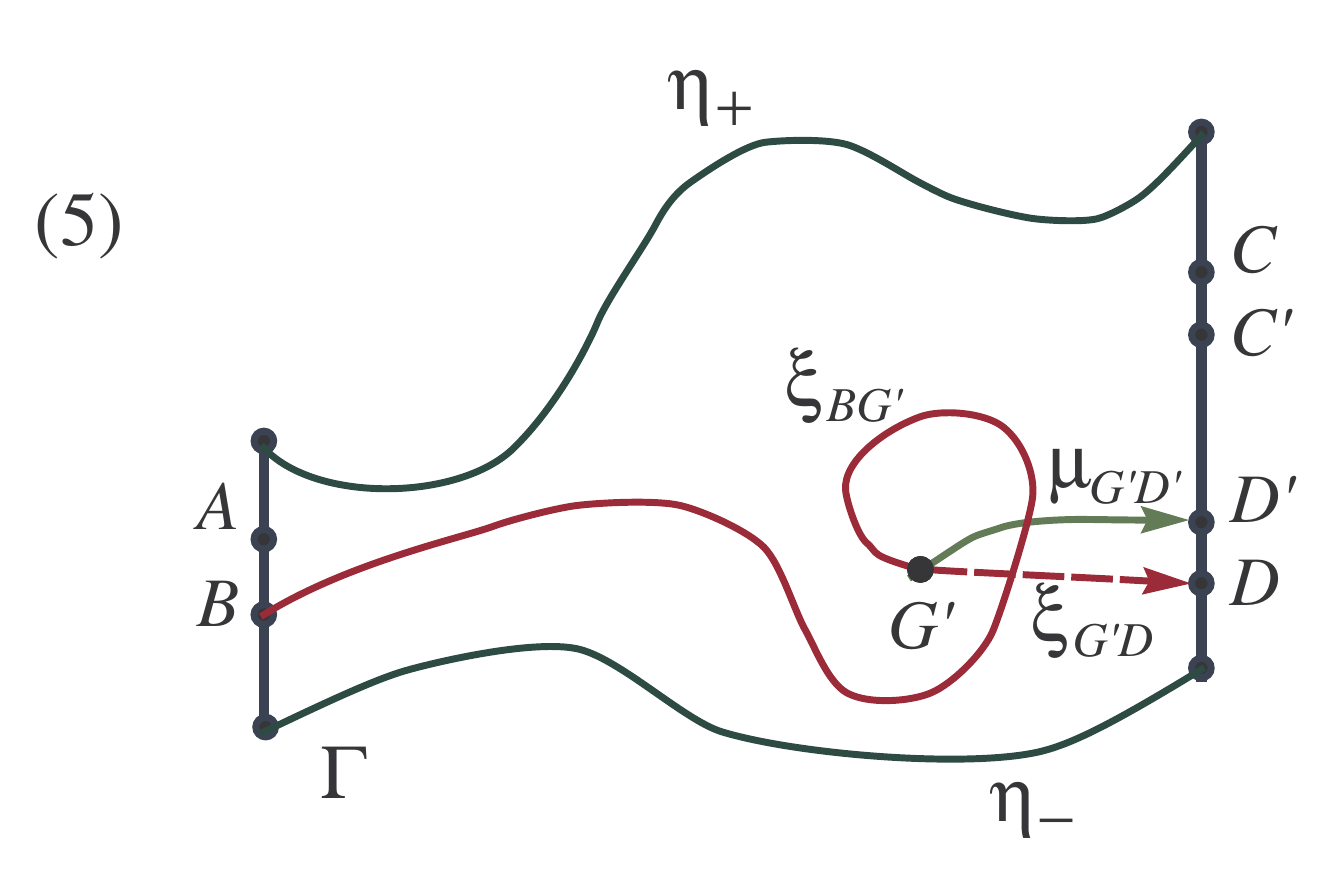}
	\caption{Steps of the construction of injection~$\Phi$ in the proof of the lemma.}
	\label{f:steps}
\bigskip

\end{figure}

\bigskip

\section{Generalizations and applications} \label{sec:gen}

\subsection{General transition probabilities}\label{ss:gen-trans}
In the notation of the introduction, consider a more general random walk \ts
$X_t$ \ts which moves to neighbors with general (not necessarily uniform)
\emph{transition probabilities}:
$$
\Pb\bigl[(i,j) \to (i\pm 1,j)\bigr] \. = \. \pi_{\pm}(i,j), \qquad
\Pb\bigl[(i,j) \to (i,j\pm 1)\bigr] \. = \. \om_{\pm}(i,j),
$$
with obvious constraints \ts $\pi_{\pm}(i,j), \ts \om_{\pm}(i,j)\ge 0$,
and
$$
\pi_{+}(i,j) \. + \. \pi_{-}(i,j)\. + \. \om_{+}(i,j) \. + \. \om_{-}(i,j) \, = 1,
$$
for all \ts $(i,j)\in \Ga$.
We say these transition probabilities are \emph{$y$-invariant} if they
are translation invariant with respect to  vertical shifts:
$$
\pi_{\pm}(i,j) \. = \pi_{\pm}(i,j'), \quad  \om_{\pm}(i,j) \. = \. \om_{\pm}(i,j') \quad
\text{for all \, $i,\ts j$ \. and \. $j'$.}
$$

\smallskip

\begin{thm}\label{t:trans}
Let \ts $\Ga\ssu \zz^2$ \ts be the lattice region as in Theorem~\ref{t:main}. Let
\ts $\{X_t\}$ \ts be the lattice random walk which
starts at the origin \ts $X_0=O \in \al$, moves according to
$y$-invariant transition probabilities \ts $\pi_{\pm}(i,j)$,
\ts $\om_{\pm}(i,j)$ \ts as above, and is absorbed
whenever \ts $X_t$ \ts tries to exit the region~$\ts\Ga$.  Denote by \ts $T$ \ts the
first time~$t$ such that \ts $X_t \in \be$, and let \ts $p(k)$ \ts be
the probability that $X_T=(m,k)$.  Then \ts $\{p(k)\}$ \ts is log-concave:
$$p(k)^2 \, \ge \, p(k+1) \. p(k-1) \qquad \text{for all \ $k\in \zz$, \
such that \ \ $(m,k\pm 1)\ts \in \be$}.
$$
\end{thm}

\smallskip

The proof of the theorem follows verbatim the proof of Theorem~\ref{t:main}
in the previous section, since $y$-invariance is exactly the property
preserved by the injection $\Phi$ in Lemma~\ref{l:main}.  \ $\sq$

\medskip

\subsection{Monotone walks}\label{ss:gen-endpoints}
Let \ts $\{X_t\}$ \ts be a random walk as above with $y$-invariant transition
probabilities.  We say that the walk is \emph{monotone} if \ts $\pi_-(i,j)=\om_-(i,j)=0$
\ts for all \ts $(i,j)\in \Ga$.
\smallskip

\begin{cor}\label{c:mono}
Let \ts $\Ga\ssu \zz^2$ \ts be the lattice region as in Theorem~\ref{t:main},
and let \ts $\ga = \{x=\ell\}\cap \Ga$ \ts be a vertical interval,
\ts $0< \ell < m$.  Fix points \ts $A=(0,a) \in \al$ \ts
and  \ts $B=(m,b)\in \be$.  Let \ts $\{X_t\}$ \ts be a monotone lattice
random walk which starts at point \ts $X_0=A \in \al$, moves according to
$y$-invariant transition probabilities \ts $\pi_{+}(i,j)$,
\ts $\om_{+}(i,j)$ \ts as above, is absorbed
whenever \ts $X_t$ \ts tries to exit the region~$\ts\Ga$, and arrives
at $B$ at time \ts $u=(m+b-a)$: \. $X_u=B$.  Denote by \ts $T$ \ts the
first time~$t$ such that \ts $X_t \in \ga$, and let \ts $q(k)$ \ts be
the probability that $X_T=(\ell,k)$.  Then \ts $\{q(k)\}$ \ts is log-concave:
$$q(k)^2 \, \ge \, q(k+1) \. q(k-1) \qquad \text{for all \ $k\in \zz$, \
such that \ \ $(\ell,k\pm 1)\ts \in \ga$}.
$$
\end{cor}

\smallskip

\begin{proof}
Define two regions: \. $\Ga_1: = \{(i,j) \in \Ga~:~0\le i \le \ell\}$ \.
and \. $\Ga_2: = \{(i,j) \in \Ga~:~\ell-1 \le i \le m\}$.  Note that the
regions are overlapping along interval~$\ga$ and interval \ts
$\ga':=\{(\ell-1,j) \in \Ga\}$, see Figure~\ref{f:tri}.
Suppose \ts $(\ell-1,k)\to (\ell,k)$ \ts
is the unique edge of the lattice path \ts $A\to B$ \ts which projects onto
\ts $[\ell-1,\ell]$.  In the notation of Theorem~\ref{t:trans}, we have
$$
q(k) \. = \. p_1(k) \. p_2(k),
$$
where \ts $p_1(k)$ \ts and \ts $p_2(k)$ \ts are the exit probabilities
in the region $\Ga_1$ and in the region $\Ga_2$ rotated~$180^\circ$.
Since \ts $\{p_i(k)\}$ \ts are log-concave by Theorem~\ref{t:trans},
so is \ts $\{q(k)\}$, as desired.
\end{proof}

\begin{figure}[hbt]
\includegraphics[width=14.5cm]{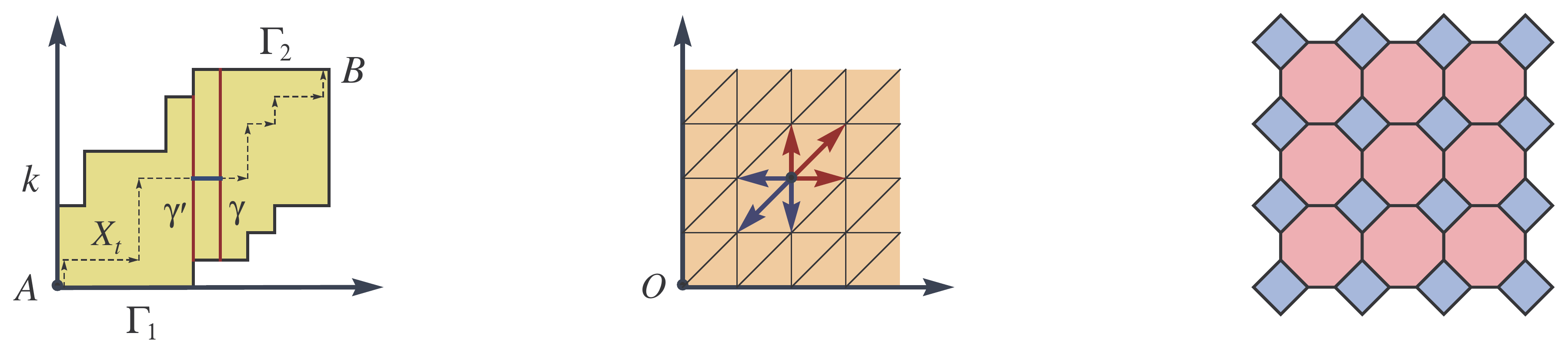}
	\caption{{\underline {Left}}: \ts a monotone walk $X_t$ crossing the vertical
$\ga$ and $\ga'$ (red lines) at height $k$.
{\underline {Middle}}: \ts steps in the triangular lattice.  {\underline {Right}}: \ts the square--octagon lattice.}
	\label{f:tri}
\end{figure}

\medskip

\subsection{General lattices}\label{ss:gen-triangular}
One can further generalize random walks from $\zz^2$ to general lattices.
For example, we can include steps \ts $\pm (1,1)$ \ts
and $y$-invariant transition probabilities
$$\Pb\bigl[(i,j) \to (i\pm 1,j\pm 1)\bigr] \. = \. \nu_{\pm}(i,j),
$$
such that \.
$\nu_{\pm}(i,j) \ts = \ts \nu_{\pm}(i,j')$, for all~$i,j$~and~$j'$.
One can view this result as a random walk on the \emph{triangular lattice} instead,
see Figure~\ref{f:tri}.  Both the statement and the proof of Theorem~\ref{t:trans}
extend verbatim once the reader observes that all lattice paths which
must intersect for topological reasons do in fact intersect at lattice points.

Similarly, one can use this approach and general transition probabilities
to set some of them zero and obtain random walks on other lattices.
For example, one can obtain the \emph{square--octagon lattice} as in the
figure by restricting the walks to vertices of the lattice.
Theorem~\ref{t:trans} applies to this case then.  We omit the details.

\medskip

\subsection{Dyck and Schr\"oder paths}\label{ss:gen-dyck}
In the context of \emph{enumerative combinatorics},
it is natural to consider lattice paths with steps \ts
$(1,1)$ \ts and \ts $(1,-1)$.  Such paths are called \defn{Dyck paths}.
When step \ts $(2,0)$ \ts is added, such paths are called \defn{Schr\"oder paths}.

Fix two points \ts $A=(0,0)$ \ts and  $B=(m,b) \in \zz^2$ \ts and two
nonintersecting Dyck paths \ts $\eta_\pm : A\to B$.  Note that \ts
$m+b \equiv 0\pmod{2}$, since otherwise there are no such paths.
Denote by $\Ga$ the region between these paths.  Let $0< \ell < m$,
and denote by \ts $\aN(k)$ \ts the number of Dyck paths \ts
\ts $\zeta: A\to B$ \ts which lie inside~$\Ga$ and contain point~$(\ell,k)$.


\smallskip

\begin{cor}  \label{c:dyck}
In the notation above, the sequence \ts $\{\aNr(k)\}$ \ts is log-concave:
$$
\aNr(k)^2 \, \ge \, \aNr(k+2) \. \aNr(k-2) \qquad \text{for all \ $k\in \zz$, \
such that \ \ $(\ell,k\pm 2)\ts \in \ga$}.
$$
\end{cor}

\smallskip

\begin{proof}[Proof sketch]
The proof follows verbatim the proof of Corollary~\ref{c:mono} via two observations. First, the vertical translation and topological properties used in the proof of Lemma~\ref{l:main} work with diagonal steps. Second, the intersection points of the paths are at the ends of the steps, not midway, because the Dyck paths here have endpoints on the underlying grid spanned by $(1,1)$ and $(1,-1)$ which is invariant under the $(0,2)$ translation.
\end{proof}

\smallskip

Similarly, fix two
nonintersecting Schr\"oder paths \ts $\eta_\pm : A\to B$, and
denote by $\Ga$ the region between these paths.  Let $0< \ell < m$,
and denote by \ts $\aF(k)$ \ts the number of Schr\"oder paths \ts
\ts $\zeta: A\to B$ \ts which lie inside~$\Ga$ and contain point~$(\ell,k)$. The same argument as above gives the following.
\smallskip

\begin{cor}  \label{c:sch}
In the notation above, the sequence \ts $\{\aFr(k)\}$ \ts is log-concave:
$$
\aFr(k)^2 \, \ge \, \aFr(k+2) \. \aFr(k-2) \qquad \text{for all \ $k\in \zz$, \
such that \ \ $(\ell,k\pm 2)\ts \in \ga$}.
$$
\end{cor}

\begin{rem}{\rm
Note that this result does not directly apply to the, otherwise similar, \defn{Motzkin paths}, with steps $(1,1)$, $(1,-1)$ and $(1,0)$. The reason is that the intersection points used in the injection $\Phi$ in Lemma~\ref{l:main} might no longer be on the underlying lattice points and appear in the middle of the steps, e.g. at $(\frac12,\frac12)$.
}
\end{rem}

\smallskip

\begin{ex}\label{ex:dyck}
{\rm Take \ts $A=(0,0)$, \ts $B=(2n,0)$, so \ts $m=2n$.  Fix maximal
and minimal Dyck paths
$$\aligned
\eta_+: \. & (0,0) \. \to \. (1,1) \. \to \. \ldots \. \to \. (n,n) \. \to \. (n+1,n-1) \. \to \. \ldots \. \to \.  (2n,0),\\
\eta_-: \. & (0,0) \. \to \. (1,-1) \. \to \. \ldots \. \to \. (n,-n) \. \to \. (n+1,-n+1) \. \to \. \ldots \. \to \.  (2n,0).
\endaligned
$$
Set \ts $\ell:=n$, which makes the picture symmetric.
Then Corollary~\ref{c:dyck} implies log-concavity of \defn{binomial coefficients} \.
$\bigl\{\binom{n}{k}, \ts 0\le k \le n\}$. On the other hand,
for the Schr\"oder paths $A\to B$,
Corollary~\ref{c:sch} implies log-concavity of \defn{Delannoy numbers} \.
$\bigl\{\rD(k,n-k), \ts 0\le k \le n\}$,
see \cite[\href{https://oeis.org/A008288}{A008288}]{OEIS}, a new enumerative
result, see $\S$\ref{ss:finrem-ec}.

Finally, let \ts $\eta_+$ \ts be as above and let
$$
\eta_-: \. (0,0) \. \to \. (1,1) \. \to \. (2,0) \. \to \. (3,1) \. \to \.\ldots \. \to \.  (2n,0).
$$
Then Corollary~\ref{c:dyck} implies log-concavity of \defn{ballot numbers} \.
$\bigl\{\rB(k,n-k), \ts n/2\le k \le n\}$, where
$$
\rB(k,n-k)\, =\,\. \frac{2k-n+1}{k+1}\ts\binom{n}{k}.
$$
}\end{ex}


\smallskip

\subsection{Boundary matters} \label{ss:gen-boundary}
First, let us note that both Theorem~\ref{t:main} and Theorem~\ref{t:trans}
easily extend to regions
without either or both of the boundaries \ts $\eta_\pm$\ts. In this case
the vertical boundaries $\al,\be$ become either rays of lines, and
the region $\Ga$ is an infinite strip in one or two directions, see Figure~\ref{f:inf}.

\smallskip

\begin{cor}\label{c:inf}
In the notation of Theorem~\ref{t:main}, let $\ts\Ga$ be an infinite strip with
one or two ends.  Then the exit probability distribution \ts $\{p(k)\}$ \ts
is log-concave.
\end{cor}

\smallskip

The proof follows immediately from Theorem~\ref{t:main} by taking a sequence
\ts $\{\Ga_n\}$ \ts of regions where the boundary goes to infinity, i.e.\ \ts
$\Ga_n \to \Ga$, and noting
that log-concavity is preserved in the limit.  Alternatively, one can easily
modify the proof of Lemma~\ref{l:main} to work for unbounded regions; in fact
the construction simplifies in that case.  We omit the details.

\smallskip

One can also ask whether the $x$-monotonicity assumption in the theorem can
be dropped.  Note that the proof of Lemma~\ref{l:main} breaks in step~$(2)$
as the shifted path \ts $\zeta_{B'C'}$ \ts no longer has to lie inside~$\Ga$,
see Figure~\ref{f:inf}.  Although we do not believe that Theorem~\ref{t:main}
extends to non-monotone boundaries, it would be interesting to find a formal
counterexample.

\begin{figure}[hbt]
\includegraphics[width=16.8cm]{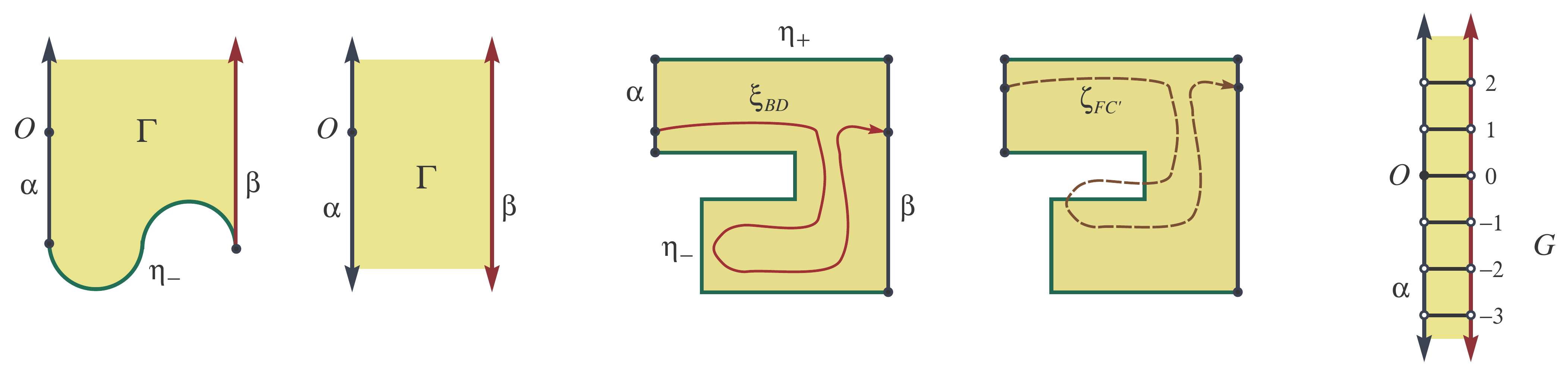}
	\caption{Infinite regions with one and two ends, the issue with non-monotone boundary
in step~$(2)$ of the proof of the lemma, and a ladder graph~$G$.}
	\label{f:inf}
\end{figure}

\begin{ex}\label{ex:strip} {\rm
In the notation above, let $m=1$ and consider an infinite strip between
two lines which forms a \defn{ladder graph}~$G$ as in Figure~\ref{f:inf}.  When restricted to~$G$,
the nearest neighbor random walk moves along~$\al$ with equal probability \ts $\frac13$ \ts of
going up or down, until it eventually moves to the right, at which point it stops.
In this case the exit probabilities can be calculated explicitly:
$$
p(\pm 2r) \, = \,\sum_{n=0}^{\infty} \. \frac{1}{3^{2n+1}} \. \binom{2n}{n-r}\,, \qquad
p(2r+1) \, = \,p(-2r-1) \, = \, \sum_{n=0}^{\infty} \. \frac{1}{3^{2(n+1)}} \. \binom{2n+1}{n-r}\,,
$$
for all \ts $r\ge 0$.  A direct calculation gives:
$$
p(\pm k)  \, = \, \frac{1}{\phi^{2k}\ts\sqrt{5}}  \quad \text{for all \, $k\ge 0$}\ts, \quad \text{where} \ \ \phi\. = \. \frac{\sqrt{5}+1}{2} \ \ \text{is the \defn{golden ratio}}.
$$
Thus, log-concavity is an equality at all $k\ne 0$.  We leave it as an exercise to the reader
to give a direct bijective proof of this fact.  Note that these equalities disappear for \ts $m \ge 2$,
cf.~$\S$\ref{ss:finrem-cauchy}.
}\end{ex}

\bigskip

\section{Final remarks} \label{sec:finrem}

\subsection{} \label{ss:finrem-posets}
Log-concavity of the number of monotone lattice paths as in Corollary~\ref{c:mono}
is equivalent to the Stanley inequality for posets of width two, as noted
in~\cite{CFG,GYY}.  For general posets, the Stanley inequality was proved
in~\cite{Sta1}.  An explicit injection in the width two case was given
in~\cite{CFG} and generalized by the authors~\cite{CPP1,CPP2}.
The construction in~\cite{CPP2} was the basis of this paper.

\subsection{} \label{ss:finrem-new}
Except for the special case of monotone lattice paths and monotone boundary
discussed above, we are not aware of the problem even being considered before.
The generality of our results is then rather surprising given that even simple
special cases appear to be new (see below).

\subsection{} \label{ss:finrem-hist}
The reflection principle is due to Mirimanoff (1923), and is often misattributed to
Andr\'e, see~\cite{Ren}. It is described in numerous textbooks, both classical~\cite{Fel,Spi}
and modern~\cite{LL,MP}.  For the Karlin--McGregor formula (1959) and
its generalizations, including the Brownian motion version of Fomin's result
(Lemma~\ref{l:Fomin}), see e.g.~\cite[Ch.~9]{LL}.  For the
Lindstr\"om--Gessel--Viennot lemma and applications to enumeration
of lattice paths, see the original paper~\cite{GV} and the extensive treatment
in~\cite[$\S$5.4]{GJ}.  It is also related to a large body of work on tilings
in the context of \emph{integrable probability}, see~\cite{Gor}.  For the
algebraic combinatorics context of Fomin's result in connection with
total positivity, see~\cite[$\S$5]{Pos}.

\subsection{} \label{ss:finrem-ineq}
Note also that the log-concavity of exit probabilities does not seem to be
a consequence of any standard non-combinatorial approaches.
For example, the \emph{real-rootedness} fails already for
Delannoy numbers \ts $\bigl\{\rD(k,8-k), \ts 0\le k \le 8\}$,
see~$\S$\ref{ss:gen-dyck}.
We refer to \cite{Bre,Sta2} for  surveys of classical methods on unimodality
and log-concavity, and to~\cite{Pak} for a short popular introduction to
combinatorial methods. See also surveys~\cite{Bra,Huh} for more recent
results and advanced algebraic and analytic tools.

\subsection{} \label{ss:finrem-ec}
In the context of Example~\ref{ex:dyck}, Dyck, Schr\"oder and Motzkin paths
play a fundamental role in enumerative
combinatorics in connection with the \emph{Catalan numbers}
\cite[\href{https://oeis.org/A000108}{A000108}]{OEIS},
\emph{Schr\"oder numbers} \cite[\href{https://oeis.org/A006318}{A006318}]{OEIS} and
\emph{Motzkin numbers} \cite[\href{https://oeis.org/A001006}{A001006}]{OEIS},
respectively.  Ballot numbers and Delannoy numbers appear in exactly the same
context.  We refer to~\cite[Ch.~5]{EC} for numerous properties of these numbers.

While binomial coefficients and ballot numbers are trivially log-concave
via explicit formulas, the log-concavity of Delannoy numbers appears to
be new.  Non-real-rootedness in this case suggests that already this special
case is rather nontrivial.  It would be interesting to see if log-concavity
of Delannoy numbers can be established directly, in the style of  basic
combinatorial proofs in~\cite{Sag}.

\subsection{} \label{ss:finrem-quarter}
There is a large literature on exact and asymptotic counting of various
walks in the quarter plane with small steps, see e.g.~\cite{Bou,BM}.
Most notably, both \emph{Kreweras walks} (1965) and \emph{Gessel walks} (2000)
fit our framework, while some others do not.  It would be
interesting to further explore this connection.

\subsection{} \label{ss:finrem-cauchy}
In the context of~$\S$\ref{ss:gen-boundary} and Example~\ref{ex:strip},
consider a simple random walk constrained to a strip \ts $0\le x \le m$,
reflected at \ts $x=0$, \ts and with no top/bottom boundaries.  This
special case is especially elegant.
The exit probabilities \ts $p(mt)$, as $m\to \infty$,
converge to the \emph{hyperbolic secant distribution}, which is log-concave
in~$t$.\footnote{See the {\tt MathOverflow} answer: \ts
\url{https://mathoverflow.net/a/395065/4040}.}
This is in sharp contrast with the case of a
simple random walk which starts at the origin, but is \emph{not} constrained to be in
the \ts $x\ge 0$ \ts halfplane.  Denote by  \ts $q(k)$ \ts the hitting probabilities of the point
$(k,m)$ on the line $x=m$.  In this case it is well known that hitting probabilities \ts $q(mt)$,
as \ts $m\to \infty$, converge to the \emph{Cauchy distribution}, see e.g.~\cite[p.~156]{Spi},
which is not log-concave (in~$t$).

\vskip.4cm

\subsection*{Acknowledgements}
We are grateful to Robin Pemantle for interesting discussions on the subject,
and to Timothy Budd for his help resolving the \ts {\tt MathOverflow} \ts
question.  Nikita Gladkov pointed out a subtle issue in step~$(3)$ of
the construction in the previous version of this paper, necessitating
the addition of Lemma~\ref{l:Fomin}.
We are also thankful to Per Alexandersson and Cyril Banderier
for helpful remarks on the paper.   We also thank the anonymous
referees for useful comments to improve presentation.
The last two authors were partially supported by the NSF.


\vskip.9cm

\end{document}